\documentclass{article}
\usepackage{amssymb,amsmath,latexsym, amsthm,enumerate, amsfonts}
\usepackage[utf8]{inputenc}
\title{Some extensions to the functional analytic approach to Colombeau algebras}
\author{Eduard A.~Nigsch}
\usepackage{hyperref}
\usepackage{xcolor}
\usepackage{mathtools}

\newtheorem{theorem}{Theorem}
\newtheorem{definition}[theorem]{Definition}

\newtheorem{lemma}[theorem]{Lemma}
\newtheorem{proposition}[theorem]{Proposition}

\newcommand{\test}[2]{\mathbf{S}(#1, #2)}
\newcommand{\dtest}[1]{\mathbf{S}(#1)}
\newcommand{\ztest}[2]{\mathbf{S^0}(#1, #2)}
\newcommand{\dztest}[1]{\mathbf{S^0}(#1)}

\DeclareMathOperator{\id}{id}
\DeclareMathOperator{\csn}{csn}

\providecommand{\abso}[1]{\left\lvert#1\right\rvert}

\newcommand{\bR}{\mathbb{R}}
\newcommand{\bE}{\mathbb{E}}
\newcommand{\bF}{\mathbb{F}}

\newcommand{\bN}{\mathbb{N}}
\newcommand{\D}{\mathrm{D}}

\let\e\varepsilon

\newcommand{\ud}{\mathrm{d}}

\newcommand{\coleq}{\coloneqq}

\newcommand{\cE}{\mathcal{E}}
\newcommand{\cF}{\mathcal{F}}
\newcommand{\cG}{\mathcal{G}}
\newcommand{\cH}{\mathcal{H}}
\newcommand{\cK}{\mathcal{K}}
\newcommand{\cL}{\mathcal{L}}
\newcommand{\cN}{\mathcal{N}}
\newcommand{\cD}{\mathcal{D}}
\newcommand{\cS}{\mathcal{S}}

\hypersetup{
    colorlinks,
    linkcolor={blue!80!black},
    citecolor={blue!80!black},
    urlcolor={blue!80!black}
}

\author{Eduard A.~Nigsch\footnote{Wolfgang Pauli Institute, Oskar-Morgenstern-Platz 1, 1090 Vienna, Austria.\newline e-mail: \href{mailto:eduard.nigsch@univie.ac.at}{eduard.nigsch@univie.ac.at}}}
\title{Some extensions to the functional analytic approach to Colombeau algebras}
\date{}

\begin{document}
\maketitle
\begin{center}
{\em Dedicated to Professor J.~A.~Vickers on the occasion of his 60th birthday} 
\end{center}
\begin{abstract}
We extend the functional analytic approach to Colombeau-type spaces of nonlinear generalized functions in order to study algebras of tempered generalized functions. We obtain a definition of Fourier transform of nonlinear generalized functions which has a strict inversion theorem, agrees with the classical Fourier transform for tempered distributions and preserves well-known classical properties.
\\[2mm] {\it AMS Mathematics  Subject Classification $(2010)$}: 46F30
\\[1mm] {\it Key words and phrases:} Colombeau algebra; tempered generalized function; Fourier transform
\end{abstract}

\section{Introduction}

As in the field of linear generalized functions (distribution theory), a concept of tempered generalized functions and their Fourier transform is essential also in the context of Colombeau algebras
(\cite{ColNew, ColElem, GKOS, MOBook, zbMATH01226424}) 
for the study of singularities, regularity theory and microlocal analysis (see, e.g., \cite{zbMATH05357102,zbMATH01384926,zbMATH01660606}). While there exist various approaches to tempered Colombeau algebras (i.e., algebras containing the space $\cS'$ of tempered distributions) and the related concept of Fourier transform (see \cite{zbMATH01384926} for an overview), in all of these the Fourier inversion theorem cannot hold in a strict sense (cf.~\cite[Remark 4.3.9]{ColElem}). Moreover, the embedding of tempered distributions commutes with the Fourier transform only in a weakened sense in these settings. In this article we will give a general construction of (full) Colombeau generalized function spaces which not only permit to have such a strict inversion theorem but also make the embedding of $\cS'$ commute with the Fourier transform.

The starting point for our construction is the functional analytic approach to Colombeau algebras developed in \cite{papernew}. This approach reflects the fact that all Colombeau algebras involve some kind of regularization of distributions (most commonly by convolution with smooth mollifiers) by the use of so-called \emph{smoothing operators}, which in general are linear continuous mappings from some space of distributions into some space of smooth functions.

In the functional analytic approach of \cite{papernew} the basic space containing the representatives of generalized functions is given by $C^\infty ( \cL_b ( \cD', C^\infty), C^\infty)$ (see Section \ref{sec_prelim} for notation). Moderate and negligible representatives are singled out by evaluating them on \emph{test objects}, which are nets $(\Phi_\e)_{\e \in (0,1]}$ in $\cL(\cD', C^\infty)$ converging to the identities in $\cL(\cD', \cD')$ and $\cL(C^\infty, C^\infty)$ and being bounded in a certain sense. In other words, if $R$ is an element of the basic space then the asymptotic behavior of $R(\Phi_\e)$ for $\e \to 0$ determines whether $R$ is moderate or negligible.

In this article we replace the smoothing operators $\Phi \in \cL ( \cD', C^\infty )$ by elements of $\cL(\cH, \cK)$ for rather arbitrary spaces of distributions $\cH$ and $\cK$ in place of $\cD'$ and $C^\infty$. This has the purpose of fine-tuning the Colombeau algebra for inclusion of subspaces of $\cD'$ and thus obtaining representatives of generalized functions which have additional properties, possibly better reflecting the properties of the embedded distributions. Moreover, we consider test objects which are not only usable for one pair $(\cH, \cK)$ but for several at the same time (this is the case, for example, with regularization by convolution with a smooth mollifier). This will allow us to consider inclusions between different Colombeau algebras. Based on this we can study tempered generalized functions in our framework and obtain a tempered Colombeau algebra $\cG_\tau$ whose Fourier transform commutes with the embedding of $\cS'$ and allows for a strict inversion theorem; furthermore, $\cG_\tau$ will be naturally contained in an algebra containing all distributions.

\section{Preliminaries} \label{sec_prelim}

As far as distribution theory is concerned we mainly follow the notation and terminology of L.~Schwartz (\cite{TD}). Given two locally convex spaces $\bE$ and $\bF$, $\cL_b(\bE, \bF)$ denotes the space of continuous linear mappings from $\bE$ to $\bF$ endowed with the topology of uniform convergence on bounded subsets of $\bE$. By $\csn(\bE)$ we denote the set of all continuous seminorms of $\bE$. $C^\infty(\bE, \bF)$ is the space of smooth mappings $\bE \to \bF$ in the sense of \cite{Froe,KM} and for $f \in C^\infty(\bE, \bF)$, $\ud f$ denotes the differential of $f$ as in \cite[3.18]{KM}.

We recall from \cite[p.~7]{zbMATH03145498} that a \emph{space of distributions} on $\bR^n$ (where $n \in \bN = \{1,2,3,\dotsc\}$ is fixed throughout the article) is a subspace $\cH \subseteq \cD'$ endowed with a locally convex topology which is finer than the topology induced by $\cD'$. A space of distributions $\cH$ is called \emph{normal} if $\cD$ is continuously included and dense in $\cH$. Note that $\cD$, $\cD'$ etc.\ always denote the corresponding spaces of functions or distributions on some open subset $\Omega \subseteq \bR^n$, i.e., $\cD = \cD(\Omega)$, $\cD' = \cD'(\Omega)$ etc.

By $I$ we denote the interval $(0,1]$ and $\id_M$ is the identity map on a set $M$. Convergence and asymptotic estimates of a net indexed by $\e \in I$ are always meant for $\e \to 0$.

\section{Test objects}

We will call \emph{test pair} a pair $(\cH, \cK)$ where $\cH$ is a normal space of distributions and $\cK$ a space of distributions which is sequentially dense in $\cH$. The most obvious example for a test pair is $(\cD', C^\infty)$. Every operator $\Phi \in \cL ( \cH, \cK)$ can be viewed as an operator in $\cL(\cD, \cD')$ with additional properties; in fact, it restricts to a map $\Phi|_{\cD} \in \cL(\cD, \cD')$ which (i) is continuous with respect to the topology induced on $\cD$ by $\cH$ and (ii) extends to a map in $\cL(\cH, \cD')$ which not only has values in $\cK$ but also is continuous into $\cK$.
Hence, we say that \emph{$\Phi \in \cL(\cD, \cD')$ is an element of $\cL(\cH, \cK)$} if it satisfies these two conditions. We will now define test objects; variants of this definition are used one way or another in virtually all Colombeau algebras.

\begin{definition}\label{def_testobjects}
For any test pair $(\cH, \cK)$, $\test{\cH}{\cK}$ is defined to be the set of all $(\Phi_\e)_\e \in \cL(\cD, \cD')^I$ such that (i) $\Phi_\e$ is an element of $\cL(\cH, \cK)$ for all $\e$, (ii) $\Phi_\e \to \id_\cH$ in $\cL_b(\cH, \cH)$, (iii) $\forall m \in \bN$ $\forall p \in \csn ( \cL_b ( \cK, \cK))$: $p(\Phi_\e|_{\cK} - \id_{\cK}) = O(\e^m)$ and (iv) $\forall p \in \csn ( \cL_b ( \cH, \cK))$ $\exists N \in \bN$: $p(\Phi_\e) = O(\e^{-N})$.

$\ztest{\cH}{\cK}$ is defined to be the set of all $(\Phi_\e)_\e \in \cL(\cD, \cD')^I$ such that (i) $\Phi_\e$ is an element of $\cL(\cH, \cK)$ for all $\e$, (ii) $\Phi_\e \to 0$ in $\cL_b(\cH, \cH)$, (iii) $\forall m \in \bN$ $\forall p \in \csn ( \cL_b ( \cK, \cK))$: $p(\Phi_\e|_{\cK}) = O(\e^m)$ and (iv) $\forall p \in \csn ( \cL_b ( \cH, \cK))$ $\exists N \in \bN$: $p(\Phi_\e) = O(\e^{-N})$.

Elements of $\test{\cH}{\cK}$ are called \emph{test objects} and elements of $\ztest{\cH}{\cK}$ $0$-test objects. Given a family $\Delta = \{ ( \cH_\delta, \cK_\delta)\}_{\delta \in J}$ of test pairs (where $J$ is any index set) we define $\dtest{\Delta} \coleq \bigcap_{\delta \in J} \test{\cH_\delta}{\cK_\delta}$ and $\dztest{\Delta} \coleq \bigcap_{\delta \in J} \ztest{\cH_\delta}{\cK_\delta}$.
\end{definition}

Note that $\dztest{\Delta}$ is a vector space and $\dtest{\Delta}$ an affine space parallel to it.
In practice one will also need to consider subsets of $\test{\cH}{\cK}$ having additional properties, for example in order to obtain the sheaf property. The prime example of a test object is obtained from the net of mollifiers which is used for the embedding of distributions into the special Colombeau algebra (cf.~\cite[Equation (1.8)]{GKOS}).

Test objects also have a decisive role in extending operations from smooth functions or distributions to elements of the Colombeau algebra. We recall that one way to do this is by fixing the regularization parameter and performing the operation on the resulting smooth function. In fact, in special Colombeau algebras this is the only possibility. In contrast, in full Colombeau algebras and especially in our setting we can also operate on the smoothing kernels themselves, which is an indispensable feature for example if we want to obtain diffeomorphism invariance (see \cite{global}). The particular operations we have in mind are pullback along diffeomorphisms, directional derivatives and the Fourier transform. Their definition in our general setting rests on the following two lemmas.

\begin{lemma}\label{sotransform}
Suppose we are given test pairs $(\cH, \cK)$ and $(\widetilde \cH, \widetilde \cK)$ and an isomorphism of topological vector spaces $f \in \cL(\cH, \widetilde \cH)$ which restricts to a topological isomorphism $f|_{\cK} \in \cL(\cK, \widetilde \cK)$. Then the map $\Phi \mapsto f^{-1} \circ \Phi \circ f$
defines a linear topological isomorphism $\cL_b ( \widetilde \cH, \widetilde \cK ) \cong \cL_b (\cH, \cK)$ which in turn induces linear isomorphisms $\test{\widetilde\cH}{\widetilde\cK} \cong \test{\cH}{\cK}$ and $\ztest{\widetilde\cH}{\widetilde\cK} \cong \ztest{\cH}{\cK}$ defined componentwise.
\end{lemma}
\begin{proof}
This follows immediately from continuity of $f$ and its inverse.
\end{proof}

For example, let $\Omega, \widetilde \Omega \subseteq \bR^n$ be open, $(\cH, \cK) = (\cD'(\Omega), C^\infty(\Omega))$, $(\widetilde \cH, \widetilde \cK) = (\cD'(\widetilde \Omega), C^\infty(\widetilde \Omega))$, $\mu \colon \Omega \to \widetilde \Omega$ a diffeomorphism and $f \coleq \mu_*$ pushforward along $\mu$. Then Lemma \ref{sotransform} simply states that the respective spaces of (0-)test objects are invariant under diffeomorphisms. Another example is the Fourier transform $\cF \in \cL(\cS', \cS')$ which restricts to $\cF \in \cL(\cS, \cS)$: the map $\Phi \mapsto \cF^{-1} \circ \Phi \circ \cF$
gives automorphisms of $\test{\cS'}{\cS}$ and $\ztest{\cS'}{\cS}$ on which the definition of the Fourier transform in Section \ref{sec_ft} rests.

For extending the directional derivative $\D_X \colon \cD' \to \cD'$ we will employ the following.

\begin{lemma}\label{lem_difftransform}Let $T \in \cL(\cH, \cH)$ with $T|_{\cK} \in \cL(\cK, \cK)$. Then the mapping $\Phi \mapsto  T \Phi \coleq T \circ \Phi - \Phi \circ T$ is linear and continuous from $\cL(\cH, \cK)$ into itself. Applied componentwise it induces mappings $$T \colon \test{\cH}{\cK} \to \ztest{\cH}{\cK}\mbox{ and }T \colon \ztest{\cH}{\cK} \to \ztest{\cH}{\cK}.$$
\end{lemma}
\begin{proof}
Again, this follows immediately from continuity of $T$ and $T|_{\cK}$.
\end{proof}

\section{Basic spaces}

One of our aims will be to obtain inclusion relations between Colombeau algebras modelling different spaces of distributions $\cH_1, \cH_2$ on different spaces of smooth functions $\cK_1, \cK_2$, respectively. In case $\cH_1 \subseteq \cH_2$ and $\cK_1 \subseteq \cK_2$ it is desirable to have a mapping $\cG_1 \to \cG_2$ between the corresponding Colombeau algebras. For instance, an algebra containing \emph{tempered} distributions should be naturally contained in an algebra containing \emph{all} distributions.
We first consider this question on the level of the basic spaces and return to it later for the quotient. Suppose we are given basic spaces $ \cE^1 \coleq C^\infty ( \cL ( \cH_1, \cK_1), \cK_1 )$ and $ \cE^2 \coleq C^\infty ( \cL ( \cH_2, \cK_2), \cK_2 )$. As the functor $C^\infty(\_,\_)$ is contravariant in the first argument and covariant in the second, for a mapping $\cE^1 \to \cE^2$ we need to come up with mappings $\cK_1 \to \cK_2$ and $\cL ( \cH_2, \cK_2) \to \cL ( \cH_1, \cK_1)$. For the first we obviously have the inclusion; for the second, the restriction of $\Phi \in \cL(\cH_2, \cK_2)$ to $\cH_1$ would be a candidate, but only if we knew that $\Phi|_{\cH_1}$ indeed was an element of $\cL ( \cH_1, \cK_1)$. This suggests to replace $\cE^2$ by $\widetilde \cE^2 \coleq C^\infty ( \cL ( \cH_2, \cK_2) \cap \cL ( \cH_1, \cK_1 ), \cK_2 )$ in order to obtain the desired mapping $\cE^1 \to \widetilde \cE^2$.
According to this motivation we give the following definition of our basic spaces.

\begin{definition}
For $\Delta = \{( \cH_\delta, \cK_\delta)\}_{\delta \in J}$ we define
\[ \cL(\Delta) \coleq \bigcap_{\delta \in J}\cL ( \cH_\delta, \cK_\delta ) = \{ \Phi \in \cL(\cD, \cD')\ |\ \forall \delta \in J: \Phi \in \cL ( \cH_\delta, \cK_\delta ) \} \]
and endow it with the projective topology with respect to the inclusions $\cL(\Delta) \subseteq \cL_b(\cH_\delta, \cK_\delta)$ for each $\delta \in J$. For any space of distributions $\cK$ such that each $\cK_\delta$ is continuously included in $\cK$ we define
\[ \cE^{\Delta, \infty} ( \cK ) \coleq C^\infty ( \cL ( \Delta) , \cK) \quad\textrm{and}\quad \cE^{\Delta, \mathrm{d}} ( \cK ) \coleq \{ R \colon \cL ( \Delta) \to  \cK \}. \]
Here, $\infty$ stands for \emph{smooth} dependence and $\mathrm{d}$ for \emph{discrete} dependence; note that $\cE^{\Delta, \infty}(\cK) \subseteq \cE^{\Delta, \mathrm{d}} (\cK)$.
For each $\delta \in J$ we define embeddings
\begin{align*}
 \iota_\delta \colon \cH_\delta &\to \cE^{\Delta,\infty} (\cK) \subseteq \cE^{\Delta, d} (\cK),\quad (\iota_\delta u)(\Phi) \coleq \Phi(u),\\
 \sigma_\delta \colon \cK_\delta &\to \cE^{\Delta,\infty} ( \cK ) \subseteq \cE^{\Delta, d} (\cK),\quad (\sigma_\delta f)(\Phi) \coleq f.
\end{align*}
Moreover, there is the embedding $\sigma \colon \cK \to \cE^{\Delta, \infty}(\cK) \subseteq \cE^{\Delta, d}(\cK)$, $(\sigma f)(\Phi) \coleq f$.
\end{definition}
For the definition of $\iota_\delta$ we used that each $u \in \cH_\delta$ defines a linear mapping $\cL_b ( \cH_\delta, \cK_\delta) \to \cK_\delta \to \cK$, $\Phi \mapsto \Phi(u)$, which is continuous and hence smooth in the sense of \cite{KM}. We will often simply write $\iota$ and $\sigma$ in place of $\iota_\delta$ and $\sigma_\delta$.

The reason that we consider the basic space both with and without smooth dependence is that both variants are useful in different situations. For a geometric formulation of the theory where one needs diffeomorphism invariance of the algebra together with a Lie derivative commuting with the embedding, one necessarily has to use the basic space with smooth dependence (\cite{global,global2,papernew}). On the other hand, in case one only requires the embedding to commute with partial derivatives along the coordinate axes, the space with discrete dependence is sufficient; it is this variant of the algebra which is more closely related to Colombeau's presentation in \cite{ColElem}.

\section{The quotient construction}

In this section let $\Delta = \{( \cH_\delta, \cK_\delta)\}_{\delta \in J}$ be fixed and $\cK$ a space of distributions with $\cK_\delta \subseteq \cK$ continuously $\forall j \in J$. For $S \subseteq \dtest{\Delta}$ and $S^0 \subseteq \dztest{\Delta}$ arbitrary (but nonempty), the following is the appropriate definition of the natural quotient construction.

\begin{definition}
We call $R \in \cE^{\Delta,\infty} ( \cK)$ $(S, S^0)$-moderate if
\begin{align*}
& \forall p \in \csn ( \cK)\ \forall l \in \bN_0\ \exists N \in \bN\ \forall (\Phi_\e)_\e \in S, (\Psi_{1,\e})_\e, \dotsc, (\Psi_{l,\e})_\e \in S^0:\\
&\qquad p ( ( \ud^l R)(\Phi_\e)(\Psi_{1,\e}, \dotsc, \Psi_{l,\e})) = O(\e^{-N}).
\end{align*}
The vector space of all $(S, S^0)$-moderate elements of $\cE^{\Delta,\infty}(\cK)$ is denoted by $\cE^{\Delta,\infty}_M (\cK; S, S^0)$.
We call $R \in \cE^{\Delta,\infty} ( \cK)$ $(S, S^0)$-negligible if
\begin{align*}
 & \forall p \in \csn ( \cK)\ \forall l \in \bN_0\ \forall m \in \bN\ \forall (\Phi_\e)_\e \in S, (\Psi_{1,\e})_\e, \dotsc, (\Psi_{l,\e})_\e \in S^0:\\
&\qquad p ( ( \ud^l R)(\Phi_\e)(\Psi_{1,\e}, \dotsc, \Psi_{l,\e})) = O(\e^m).
\end{align*}
The vector space of all $(S, S^0)$-negligible elements of $\cE^{\Delta, \infty}(\cK)$ is denoted by $\cN^{\Delta,\infty} ( \cK; S, S^0)$ and is a linear subspace of $\cE^{\Delta, \infty}_M ( \cK; S, S^0)$.

We define $S$-moderate and $S$-negligible elements of $\cE^{\Delta, d}$ by the same conditions but with $l=0$, i.e., without derivatives with respect to the test objects. Finally, we define the quotient vector spaces
\begin{align*}
\cG^{\Delta, \infty} ( \cK; S, S^0) \coleq \frac{\cE^{\Delta, \infty}_M(\cK; S, S^0)}{\cN^{\Delta, \infty} (\cK; S, S^0)}
\quad\textrm{and}\quad
 \cG^{\Delta, d} ( \cK; S) \coleq \frac{\cE^{\Delta, d}_M(\cK; S)}{\cN^{\Delta, d} (\cK; S)}.
\end{align*}
\end{definition}

It is easy to see that $\iota_\delta(\cH_\delta) \cup \sigma_\delta(\cK_\delta) \subseteq \cE^{\Delta, \infty}_M (\cK; S, S^0)$, $(\iota_\delta - \sigma_\delta)(\cK_\delta) \subseteq \cN^{\Delta, \infty}(\cK; S, S^0)$ and $\iota_\delta(\cH_\delta) \cap \cN^{\Delta, d}(\cK; S) = \{ 0 \}$.
Further properties depend on the exact choice of $S$ and $S^0$.
Note that because $\cE_M^{\Delta, \infty}(\cK; S, S^0) \subseteq \cE_M^{\Delta, d}(\cK; S)$ and $\cN^{\Delta, \infty}(\cK; S, S^0) \subseteq \cN^{\Delta, d}(\cK; S)$ we have an induced canonical mapping $\cG^{\Delta, \infty}(\cK; S, S^0) \to \cG^{\Delta, d}(\cK; S)$ which is injective on $\iota_\delta(\cH_\delta)$ for each $\delta$.

As in all Colombeau-type generalized function spaces we have a concept of association which we mention for completeness:
\begin{definition}Let $\cH$ be a space of distributions. We say that two elements $R,S$ of $\cE_M^{\Delta, \infty}(\cK; S, S^0)$ or $\cE_M^{\Delta, d} (\cK; S)$ are $\cH$-associated if for all $(\Phi_\e)_\e \in S$, $R(\Phi_\e) - S(\Phi_\e) \to 0$ in $\cH$; in this case we write $R,S \approx_{\cH} S$. Moreover, we say that $R$ admits $u \in \cH$ as an $\cH$-associated distribution if $R(\Phi_\e) \to u$ in $\cH$ for all $(\Phi_\e)_\e \in S$.
\end{definition}
Obviously every element of $\cN^{\Delta, d}(\cK)$ is $\cH_\delta$-associated to $0$ for each $\delta \in J$, hence association is well-defined on the quotient.
The next proposition enables us to extend operations from smooth functions to generalized functions componentwise.
\begin{proposition}\label{prop_compext}
Let $k \in \bN$.
For each $i=0 \dotsc k$ let $\cK_i$ be a space of distributions and $\Delta_i = \{ ( \cH^i_{\delta}, \cK^i_{\delta} ) \}_{\delta \in J_i}$ a family of test pairs
such that $\cK^i_{\delta} \subseteq \cK_i$ continuously for all $\delta \in J_i$. Moreover, assume that $\Delta_i$ is a subfamily of $\Delta_0$ for $i=1 \dotsc k$. Then any continuous multilinear mapping $T \colon \cK_1 \times \dotsc \times \cK_k \to \cK_0$ defines multilinear mappings
\begin{align*}
 T &\colon \cE^{\Delta_1, d} ( \cK_1) \times \dotsc \times \cE^{\Delta_k, d} ( \cK_k) \to \cE^{\Delta_0, d} ( \cK_0) \\
T & \colon \cE^{\Delta_1, \infty} ( \cK_1) \times \dotsc \times \cE^{\Delta_k, \infty} (\cK_k) \to \cE^{\Delta_0, \infty} (\cK_0)
\end{align*}
given by $T ( R_1, \dotsc, R_k ) ( \Phi ) \coleq T ( R_1(\Phi), \dotsc, R_k ( \Phi) )$.

Let $\cE^{\Delta_i, d}(\cK_i)$ and $\cE^{\Delta_i, \infty}(\cK_i)$ be endowed with sets of test objects $S_i$ and $S_i^0$ such that $S_0 \subseteq \bigcap_{i=1}^k S_i$ and $S_0^0 \subseteq \bigcap_{i=1}^k S_i^0$.
Then these mappings preserve moderateness, and $T(R_1, \dotsc, R_k)$ is negligible if at least one of the $R_i$ is negligible. $T$ commutes with the respective $\sigma$-embeddings, i.e., for $\delta_i \in J_i$ ($i=1 \dotsc k$) we have
\[ T(\sigma_{\delta_1}(f_1), \dotsc, \sigma_{\delta_k}(f_k)) = \sigma(T(f_1,\dotsc,f_k)). \]
If, moreover, $T$ has a sequentially continuous extension $T \colon \cH^1_{\delta_1} \times \dotsc \cH^k_{\delta_k} \to \cH^0_{\delta_0}$ then $T$ commutes with the respective $\iota$-embeddings on the level of association, i.e., for $u_i \in \cH^i_{\delta_i}$ for $i=1\dotsc k$ we have
\[ T(\iota_{\delta_1}(u_1), \dotsc, \iota_{\delta_k}(u_k)) \approx_{\cH^0_{\delta_0}} \iota_{\delta_0}(T(u_1, \dotsc, u_k)). \]
\end{proposition}

\begin{proof}This follows from the usual seminorm estimates for continuous multilinear mappings.
\end{proof}

Examples for the use of Proposition \ref{prop_compext} can be seen in the extension of the multiplication $C^\infty \times C^\infty \to C^\infty$, the convolution $\cS \times \cS \to \cS$ and the derivative $\D_X \colon C^\infty \to C^\infty$ where $X$ is a vector field on $\bR^n$.

Next, we employ Lemma \ref{sotransform} for extending operations to nonlinear generalized functions in a different way.
\begin{definition}\label{acht}
With $\Delta = \{ ( \cH, \cK) \}$ and $\widetilde \Delta = \{ ( \widetilde \cH, \widetilde \cK ) \}$ let $f \in \cL(\cH, \widetilde \cH)$ be an isomorphism which restricts to an isomorphism $\cL(\cK, \widetilde \cK)$. We define mappings $f \colon \cE^{\Delta, d}(\cK) \to \cE^{\widetilde \Delta, d}(\widetilde \cK)$ and $f \colon \cE^{\Delta, \infty}(\cK) \to \cE^{\widetilde \Delta, d}(\widetilde \cK)$ by
\begin{equation*}
(f R )(\widetilde \Phi) \coleq f ( R ( f^{-1} \circ \widetilde \Phi \circ f )) \qquad ( \widetilde \Phi \in \cL ( \widetilde \cH, \widetilde \cK ) ).
\end{equation*}
\end{definition}
The following is a direct consequence of the definitions.
\begin{lemma}In the situation of Definition \ref{acht} let $S \subseteq \dtest{\Delta}$, $S^0 \subseteq \dztest{\Delta}$ and $\widetilde S \subseteq \{ (\Phi_\e)_\e \in \dtest{\widetilde \Delta}\ |\ (f^{-1} \circ \Phi_\e \circ f)_\e \in S \}$, $\widetilde S^0 \subseteq \{ (\Phi_\e)_\e \in \dztest{\widetilde \Delta}\ |\ (f^{-1} \circ \Phi_\e \circ f)_\e \in S^0 \}$. Then we have inclusions $f(\cE^{\Delta, d}_M(\cK; S)) \subseteq \cE^{\widetilde \Delta, d}_M(\widetilde \cK; \widetilde S)$, $f(\cN^{\Delta, d}(\cK; S)) \subseteq \cN^{\widetilde \Delta, d}(\widetilde \cK; \widetilde S)$, $f(\cE^{\Delta, \infty}_M (\cK; S,S^0)) \subseteq \cE^{\widetilde \Delta, \infty}_M (\widetilde \cK; \widetilde S, \widetilde S^0)$ and $f(\cN^{\Delta, \infty}(\cK; S,S^0)) \subseteq \cN^{\widetilde \Delta, \infty}(\widetilde \cK; \widetilde S, \widetilde S^0)$. In other words, $f$ preserves moderateness and negligibility, hence is well-defined on $\cG^{\Delta, d}(\cK; S)$ and $\cG^{\Delta, \infty}(\cK; S, S^0)$. Moreover, it commutes with $\iota$ and $\sigma$.
\end{lemma}
Taking for $f$ the pushforward along a diffeomorphism one obtains diffeomorphism invariant algebras (cf.~\cite{global, papernew}). This depends essentially on the preservation of the spaces of test objects under $f$; this is not the case for the algebra $\cG^e$ (\cite[Section 1.4]{GKOS}), which fails to be diffeomorphism invariant for this reason. We remark that the construction of the first diffeomorphism invariant full Colombeau algebra constituted a major unsolved problem for several years.

In order to obtain a geometric directional derivative one takes for $f$ the flow along a (complete) vector field $X$ and differentiates the pullback along the flow at time $t=0$; this gives the following formula for spaces with smooth dependence on $\Phi$:
\[
 (\widehat \D_X R)(\Phi) \coleq -(\ud R)(\Phi)(\D_X \Phi) + \D_X ( R(\Phi) ), \\
\]
where $\D_X \Phi $ is defined as in Lemma \ref{lem_difftransform}. This is a special case of the following:
\begin{definition}\label{defblah}
 Let $\Delta = \{ ( \cH, \cK ) \}$ and $T \in \cL(\cH, \cH)$ with $T|_{\cK} \in \cL(\cK, \cK)$. Then for $R \in \cE^{\Delta, \infty}(\cK)$ we define $TR \in \cE^{\Delta, \infty}(\cK)$ by
 \[ (TR)(\Phi) \coleq T(R(\Phi)) - \ud R (\Phi) (T \circ \Phi - \Phi \circ T). \]
\end{definition}

\begin{lemma}
In the situation of Definition \ref{defblah} let $S \subseteq \dtest{\Delta}$ and $S^0 \subseteq \dztest{\Delta}$ with $T(S) \cup T(S^0) \subseteq S^0$. Then $T \colon \cE^{\Delta, \infty}(\cK) \to \cE^{\Delta, \infty}(\cK)$ preserves $(S,S^0)$-moderateness and $(S,S^0)$-negligibility and hence is defined also on the quotient $\cG^{\Delta, \infty}(\cK; S, S^0)$. Moreover, it commutes with the embeddings $\iota$ and $\sigma$.
\end{lemma}

In principle one can
generalize Definitions \ref{acht} and \ref{defblah} to the case where $\Delta$ and $\widetilde \Delta$ consist of more than one test pair, but then one has to require that $\Phi \mapsto f^{-1} \circ \Phi \circ f$ maps $\cL(\widetilde\Delta)$ into $\cL(\Delta)$ in the first case and that $T$ maps $\cL(\Delta)$ into itself in the second case. If this cannot be achieved one possibly has to change the domain of the basic space to something more general, as will be necessary in our case study of tempered generalized functions in Section \ref{sec_ft}.

For $\Delta \subseteq \widetilde \Delta$ and $\cK \subseteq \widetilde \cK$ we have canonical mappings $\cE^{\Delta,\mathrm{d}} ( \cK) \to \cE^{\widetilde{\Delta},\mathrm{d}}(\widetilde \cK)$ and $\cE^{\Delta,\infty} ( \cK) \to \cE^{\widetilde{\Delta}, \infty}(\widetilde \cK)$ given by $R \mapsto R|_{\cL(\widetilde \Delta)}$.

\begin{proposition}\label{prop_incl} Let $\Delta \subseteq \widetilde \Delta$, $\cK \subseteq \widetilde \cK$ continuously and let $\mu$ be the canonical mapping $\cE^{\Delta,\infty} ( \cK) \to \cE^{\widetilde \Delta,\infty} ( \widetilde \cK)$.
 \begin{enumerate}[(a)]
  \item If $\widetilde S \subseteq S$ and $\widetilde S^0 \subseteq S^0$, $\mu$ maps $\cE^{\Delta,\infty}_M(\cK; S, S^0)$ into $\cE^{\widetilde \Delta,\infty }_M(\widetilde \cK; \widetilde S, \widetilde S^0)$ and $\cN^{\Delta,\infty}(\cK; S, S^0)$ into $\cN^{\widetilde \Delta, \infty}(\widetilde \cK; \widetilde S, \widetilde S^0)$ and thus gives a well-defined map $\cG^{\Delta,\infty} (\cK; S, S^0) \to \cG^{\widetilde \Delta,\infty} (\widetilde \cK; \widetilde S, \widetilde S^0)$.
 \item If $\widetilde S = S$ and $\widetilde S^0 = S^0$, the map $\cG^{\Delta,\infty} (\cK; S, S^0) \to \cG^{\widetilde \Delta,\infty} (\widetilde \cK; \widetilde S, \widetilde S^0)$ is injective.
 \end{enumerate}
Similar statements hold for discrete dependence.
\end{proposition}

Again, this is easily verified. Proposition \ref{prop_incl} suggests that it may be worthwile to find classes of test objects which are test objects simultaneously for many pairs $(\cH, \cK)$.

\section{Tempered generalized functions}\label{sec_ft}

We will now outline how an algebra of tempered generalized functions on $\bR^n$ can be constructed in various ways, depending on the requirements, using the ideas presented above. By $\cF$ and $\cF^{-1}$ we denote the Fourier transform on $\cS'$ and its inverse, respectively.

The following construction depends on the fact that $S = \test{\cD'}{C^\infty} \cap \test{\cS'}{\cS}$ is nonempty; while we will not prove this in detail here, an element of this space can be obtained as follows: choose $\psi \in \cD$ with $\psi(x) = 1$ for $\abso{x} \le 1$ and $\psi(x) = 0$ for $\abso{x} \ge 2$. Set $\psi^\e(x) \coleq \psi(\e x)$, $\varphi = \cF^{-1}(\psi) \in \cS$ and $\varphi_\e(y) \coleq \cF^{-1} ( \psi^\e ) (y) = \e^{-n} \varphi (y/\e)$ as well as $\psi_\e(x) \coleq \psi(x/\e)$. The desired mapping $\Phi_\e \in \cL(\cD_y, \cD'_x)$ then is defined via the Schwartz kernel theorem by the kernel $(x,y) \mapsto \psi^\e(x) \varphi_{\e^2}(y-x) \psi_\e(y-x)$, and we have $(\Phi_\e)_\e \in S$.

As a first step consider the algebra $\cG_\tau \coleq \cG^{\Delta, \infty}(\cS; S, S^0)$ with $\Delta = \{ ( \cS', \cS) \}$, $S = S(\Delta)$ and $S^0 = S^0(\Delta)$. It is an associative commutative differential algebra containing $\cS'$ with componentwise product and convolution, as well as derivations $\D_X$ (Proposition \ref{prop_compext}) and $\widehat \D_X$ (Definition \ref{defblah}); $\sigma$ is an algebra embedding. Defining the Fourier transform of its elements using Definition \ref{acht}, i.e., $(\cF R)(\Phi) \coleq \cF ( R ( \cF^{-1} \circ \Phi \circ \cF))$, we obtain the following properties:
\begin{enumerate}[(i)]
 \item $\cF \colon \cG_\tau \to \cG_\tau$ is  a linear isomorphism whose inverse is given by the mapping $(\cF^{-1}R)(\Phi) \coleq \cF^{-1}(R(\cF \circ \Phi \circ \cF^{-1}))$.
 \item $\cF$ and $\cF^{-1}$ commute with the embeddings $\iota$ and $\sigma$.
 \item $\cF(\tau_a R) = \chi_{-a} \cF(R)$, $\cF(\chi_a R) = \tau_a \cF(R)$ where for $a \in \bR^n$, $\tau_a$ denotes translation by $a$ and $\chi_a$ is the map $x \mapsto \exp (2\pi i a x)$; similarly for $\hat \tau_a$ and $\hat \chi_a$ defined via Definition \ref{acht}.
 \item $D^\alpha (\cF(R)) = \cF((-2\pi i M)^\alpha R)$, $(2 \pi i M)^\beta \cF(R) = \cF(D^\beta R)$, where $M^\alpha$ denotes the function $x \mapsto x^\alpha$ ($\alpha \in \bN_0^n$); similarly for $\widehat D^\alpha$ and $\widehat M$ defined via Definition \ref{defblah}. 
 \item $\cF^{-1}(uv) = \cF^{-1}(u) * \cF^{-1}(v)$, $\cF(u * v ) = (\cF u) \cdot (\cF v)$.
\end{enumerate}

Note that these properties hold strictly and not only `in the sense of generalized tempered distributions' as is the case for some of them in \cite{zbMATH01384926}.

One may modify $\cG_\tau$ by choosing $\Delta = \{ (\cD', C^\infty), (\cS', \cS) \}$ and hence obtain an inclusion $\cG_\tau \subseteq \cG \coleq \cG^{\Delta, \infty}(C^\infty; S, S^0)$. The latter algebra contains all distributions.
Because for $\Phi \in \cL(\cS', \cS)$ the property $\Phi \in \cL(\cD', C^\infty)$ is not preserved by the Fourier transform one has to replace the codomain of $\cF \colon \cG_\tau \to \cG_\tau$ by a different algebra $\widehat \cG_\tau$ obtained by taking as basic space mappings from the set $\{ \cF \circ \Phi \circ \cF^{-1}\ |\ \Phi \in \cL(\Delta)\}$ into $\cS$ and using as test objects the set $\{ ( \cF \circ \Phi_\e \circ \cF^{-1})_\e \ |\ (\Phi_\e)_\e \in S\}$ and similar for $S^0$. All the above properties of the Fourier transform will be preserved, but now one has to distinguish between the ``spatial'' domain $\cG_\tau$ and the ``frequency'' domain $\widehat\cG_\tau$, and only the former one is embedded into $\cG$. In other words, one can either have equal spaces for the spatial and the frequency domain, or one can have that the first of these is embedded in a bigger algebra containing all distributions.

Concluding, we see that a careful choice of test objects for different spaces of Colombeau-type nonlinear generalized functions enables us to obtain a Fourier transform of these functions with all desirable classical properties. An in-depth study of these spaces as well as applications to regularity theory will be published in a forthcoming article.

{\bfseries Acknowledgements.} This work was supported by Project P26859 of the Austrian Science Fund (FWF).

\newcommand{\bibarxiv}[1]{arXiv: \href{http://arxiv.org/abs/#1}{\texttt{#1}}}
\newcommand{\bibdoi}[1]{{\sc doi:} \href{http://dx.doi.org/#1}{\texttt{#1}}}

\end{document}